\documentclass[12pt, a4paper, leqno]{amsart}
\usepackage{amsmath}
\usepackage{amsfonts}
\usepackage{amssymb}
\usepackage{amsthm}
\usepackage[latin1]{inputenc}
\usepackage{url}
\usepackage[active]{srcltx}

\newcommand{\R}{\mathbb{R}}

\DeclareMathOperator*{\esssup}{ess\,sup}
\DeclareMathOperator*{\essinf}{ess\,inf}

\DeclareMathOperator*{\Lip}{Lip}

\def\vint_#1{\mathchoice%
  {\mathop{\kern 0.2em\vrule width 0.6em height 0.69678ex depth
      -0.58065ex \kern -0.8em \intop}\nolimits_{\kern -0.4em#1}}%
  {\mathop{\kern 0.1em\vrule width 0.5em height 0.69678ex depth
      -0.60387ex \kern -0.6em \intop}\nolimits_{#1}}%
  {\mathop{\kern 0.1em\vrule width 0.5em height 0.69678ex depth
      -0.60387ex \kern -0.6em \intop}\nolimits_{#1}}%
  {\mathop{\kern 0.1em\vrule width 0.5em height 0.69678ex depth
      -0.60387ex \kern -0.6em \intop}\nolimits_{#1}}}

\newcommand{\art}[6]{{\sc #1, \rm #2, \it #3 \bf #4 \rm (#5), \mbox{#6}.}}
\newcommand{\book}[3]{{\sc #1, \it #2, \rm #3.}}
\newcommand{\AND}{{\rm and }}
\newcommand{\p}{{$p\mspace{1mu}$}}

\newcommand{\grad}{\nabla}
\newcommand{\Om}{\Omega}
\newcommand{\loc}{_{\rm loc}}

\newcommand{\eps}{\varepsilon}

\DeclareMathOperator*{\supp}{supp}

\theoremstyle{plain}
\newtheorem{theorem}[equation]{Theorem}
\newtheorem{lemma}[equation]{Lemma}

\numberwithin{equation}{section}

\theoremstyle{definition}

\newtheorem{deff}[equation]{Definition}

\theoremstyle{remark}
\newtheorem{remark}[equation]{Remark}

\title{On the Harnack inequality for parabolic minimizers in
  metric measure spaces}

\author{Niko Marola and Mathias Masson} \address[N.M.]{University of
  Helsinki, Department of Mathematics and Statistics, P.O. Box 68,
  FI-00014 University of Helsinki, Finland}
\email{niko.marola@helsinki.fi}

\address[M.M.]{Aalto University, Department of Mathematics, P.O. Box
  11100, FI-00076 Aalto University, Finland}
  \email{mathias.masson@aalto.fi}

\begin{document}



\begin{abstract} In this note we consider problems related to
  parabolic partial differential equations in geodesic metric measure spaces,
  that are equipped with a doubling measure and a Poincar\'e
  inequality. We prove a location and scale invariant Harnack
  inequality for a minimizer of a variational problem related to
 a doubly non-linear parabolic equation involving the
  \p-Laplacian. Moreover, we prove the sufficiency of the
  Grigor'yan--Saloff-Coste theorem for general $p>1$ in geodesic metric
  spaces. The approach used is strictly variational, and hence
  we are able to carry out the argument in the metric setting.
\end{abstract}

\maketitle

\section{Introduction}
\label{sect:Intro}

The purpose of this note is to study parabolic minimizers, which in
the Euclidean case are related  to the doubly non-linear parabolic equation
\begin{equation} \label{equation} \frac{\partial (|u|^{p-2}u)}{\partial
    t}-\grad\cdot(|\grad u|^{p-2}\grad u)=0,
\end{equation}
where $1<p<\infty$. When $p=2$ we can
recover the heat equation from \eqref{equation}. A function $u :
\Omega \times (0,T) \rightarrow \R$, where $\Omega\subset\R^n$ is a non-empty
open set, is called a parabolic quasiminimizer related to the equation
\eqref{equation} if it satisfies
\begin{align*}
p\int_{\supp(\phi)} |u|^{p-2}u \frac{\partial \phi}{\partial t} \,dx \,dt +
\int_{\supp(\phi)}&|\nabla u|^p\, dx\,dt\\
 &\leq K\int_{\supp(\phi)}|\nabla (u+\phi)|^p\, dx\,dt
\end{align*}
for some $K\geq 1$ and every smooth compactly supported function
$\phi$ in $\Omega \times (0,T)$. More precisely, in the Euclidean
setting every weak solution to \eqref{equation} is a parabolic
minimizer, i.e., a parabolic quasiminimizer with $K=1$.

Elliptic quasiminimizers were introduced by Giaquinta and Giusti in
\cite{GiaGiu82OnT, GiaGiu84Qua}. They enable the study of elliptic
problems, such as the \p-Laplace equation and \p-harmonic functions,
in metric measure spaces under the doubling property and a Poincar\'e
inequality. We refer, e.g., to \cite{BjoA}, \cite{BjoAMaro06},
\cite{BjoJ}, \cite{KinMar03Pot}, \cite{KiSha}, and the references in
these papers.  Following Giaquinta--Giusti, Wieser~\cite{wieser}
generalized the notion of quasiminimizers to the parabolic setting in
Euclidean spaces. Parabolic quasiminimizers have also been studied by
Zhou~\cite{Zho93OnT, Zho94Par}, Gianazza--Vespri~\cite{GiaVes06Par},
Marchi~\cite{Marchi}, and Wang~\cite{Wang}. The literature for
parabolic quasiminimizers is very small compared to the elliptic
case. In recent papers \cite{KiMaMiPa}, \cite{MassSilj}, parabolic
quasiminimizers related to the heat equation have been studied in
general metric measure spaces. The variational approach taken in these
papers opens up a possibility to develop a systematic theory for
parabolic problems in this generality.

Our main result is a scale and location invariant Harnack inequality,
Theorem~\ref{thm:Harnack}, in geodesic metric measure spaces for a
positive parabolic minimizer that is locally bounded away from zero
and locally bounded. We assume the measure to be doubling and to
support a $(1,p)$-Poincar\'e inequality. We take a purely variational
approach and prove the Harnack inequality without making any reference
to the equation \eqref{equation}.

In Euclidean spaces, the Harnack inequality for a positive weak
solution to the equation~\eqref{equation}, that is bounded away from
zero, was proved in \cite{KiKu}. Their proof is based on Moser's
method and on an abstract lemma due to Bombieri and Giusti. The
argument in \cite{KiKu} relies on the equation and uses, for instance,
the fact that if $u$ is a weak supersolution to \eqref{equation}, then
$u^{-1}$ is a weak subsolution of the same equation.

Our proof is based on the one in \cite{KiKu}. However, since we deal
with parabolic minimizers and upper gradients in the metric setting,
changes in the argument are required. To give an example, in the
strictly variational setting it is not true that if $u$ is a parabolic
superminimizer, then $u^{-1}$ is a parabolic subminimizer. Instead we
establish the required estimates separately for both super- and
subminimizers.

Grigor'yan~\cite{Grigor} and Saloff-Coste~\cite{SaloffIMRN} observed
independently that the doubling property and a Poincar\'e inequality
for the measure are sufficient and necessary conditions for a scale
and location invariant parabolic Harnack inequality for solutions to
the heat equation ($p=2$) on Riemannian manifolds. Later,
Sturm~\cite{Sturm} generalized this result to the setting of Dirichlet
spaces.

One motivation for the present note is to show the sufficiency for
general $1<p<\infty$ in geodesic metric measure spaces without
invoking Dirichlet spaces or the Cheeger derivative structure for
which we refer to \cite{Cheeger}. We also refer to a recent paper
\cite{BarGriKum} and to \cite{BarBasKum} on parabolic Harnack
inequalities on metric measure spaces with a local regular Dirichlet
form. It would be very interesting to know whether also the necessity
holds in this general setting.

Very recently a similar question has been studied for degenerate
parabolic quasilinear partial differential equations in the
subelliptic case by Caponga, Citti, and Rea~\cite{CaCiRe}.  Their
motivating example is a class of subelliptic operators associated to a
family of H\"ormander vector fields and their Carnot-Carath\'eodory
distance. The setup in the present paper cover also Carnot groups and
more general Carnot--Carath\'eodory spaces.

\section{Prelimininaries}
\label{sect:setup}

 In this section we briefly recall the basic definitions and collect
some results we will need in the sequel. For a more detailed treatment
we refer, for instance, to a monograph by A. and
J. Bj\"orn~\cite{BjoBjo} and to Heinonen~\cite{Heinonen}, and the
references therein.

\subsection{Metric measure spaces}
Standing assumptions in this paper are as follows. By the triplet
$(X,d,\mu)$ we denote a complete geodesic metric space $X$, where $d$ is
the metric and $\mu$ a Borel measure on $X$.  The measure $\mu$ is
supposed to be \emph{doubling}, i.e., there exists a constant $C_\mu
\geq 1$ such that
\begin{equation} \label{doubling}
0<\mu(B(x,2r)) \leq C_\mu \mu(B(x,r))<\infty
\end{equation}
for every $r > 0$ and $x \in X$. Here $B(x,r):=\{y\in X :\;
d(y,x)<r\}$. We denote $\lambda B = B(x,\lambda r)$ for each $\lambda
>0$. We want to mention in passing that to require the measure of
every ball in $X$ to be positive and finite is anything but
restrictive; it does not rule out any interesting
measures. Equivalently, for any $x\in X$, we have
\begin{equation} \label{estimatedoub}
\frac{\mu(B(x,R))}{\mu(B(x,r))} \leq C\left(\frac{R}{r}\right)^{q_\mu}
\end{equation}
for all $0 < r\leq R$ with $q_\mu := \log_2 C_\mu$, where $C>0$ is a
constant which depends only on $C_\mu$. The choice $q_\mu = \log_2 C_\mu$
is not necessarily optimal; the exponent $q_\mu$ serves as a counterpart in
metric measure space to the dimension of a Euclidean space. In
addition to the doubling property, we assume that $X$ supports a
\emph{weak $(1,p)$-Poincar\'e inequality} (see below). Moreover, the
product measure in the space $X\times (0,T)$, $T>0$, is denoted by
$\nu=\mu \otimes \mathcal{L}^1$, where $\mathcal{L}^1$ is the one
dimensional Lebesgue measure.

It is worth noting that our abstract setting causes some, perhaps
unexpected, difficulties. For instance, in not too pathological metric
spaces, it may happen that $B(x_1,r_1)\subset B(x_2,r_2)$ but
$B(x_2,2r_2)\subset B(x_1,2r_1)$.

We follow Heinonen and Koskela~\cite{HeKoActa} in introducing upper
gradients as follows. A Borel function $g: X\to [0, \infty]$ is said
to be an \emph{upper gradient} for an extended real-valued function
$u$ on $X$ if for all paths $\gamma: [0,l_{\gamma}] \to X$, we have
\begin{equation}\label{defUG}
| u(\gamma(0)) - u(\gamma(l_\gamma))| \leq \int_{\gamma} g \, ds.
\end{equation}

If \eqref{defUG} holds for \p-almost every path in the sense of
Definition~2.1 in Shanmugalingam~\cite{Sha}, we say that $g$ is a
\p-weak upper gradient of $u$. From the definition, it follows
immediately that if $g$ is a \p-weak upper gradient for $u$, then $g$
is a \p-weak upper gradient also for $u-k$, and $|k|g$ for $ku$, for
any $k\in \R$.

The \p-weak upper gradients were introduced in Koskela--MacManus
\cite{KoMac}. They also showed that if $g \in L^p(X)$ is a \p-weak
upper gradient of $u$, then one can find a sequence
$\{g_j\}_{j=1}^\infty$ of upper gradients of $u$ such that $g_j \to g$
in $L^p(X)$. If $u$ has an upper gradient in $L^p(X)$, then it has a
\emph{minimal \p-weak upper gradient} $g_u \in L^p(X)$ in the sense
that for every \p-weak upper gradient $g \in L^p(X)$ of $u$, $g_u \leq
g$ a.e. (see Shan\-mu\-ga\-lin\-gam~\cite[Corollary~3.7]{Sha-harm}).

Let $\Omega$ be an open subset of $X$ and $1\leq p<\infty$.  Following
Shanmugalingam~\cite{Sha} (see also \cite[Corollary 2.9]{BjoBjo}), we
define for $u\in L^p(\Omega)$,
\[
\|u\|_{N^{1,p}(\Omega)}^p=\|u\|_{L^p(\Omega)}^p+ \|g_u\|_{L^{p}(\Omega)}^p.
\]
The \emph{Newtonian space} $N^{1,p}(\Omega)$ ($\subset L^p(\Omega)$)
is the quotient space
\[
N^{1,p}(\Omega)=\left\{
u\in L^{p}(\Omega):\; \|u\|_{N^{1,p}(\Omega)}<\infty
\right\} / \sim,
\]
where $u\sim v$ if and only if $\|u-v\|_{N^{1,p}(\Omega)}=0$.  The
space $N^{1,p}(\Omega)$ is a Banach space and a lattice (see
Shanmugalingam~\cite{Sha}). If $u,v\in N^{1,p}(\Omega)$ and $u=v$
$\mu$-a.e., then $u\sim v$. However, if $u\in N^{1,p}(\Omega)$, then
$u\sim v$ if and only if $u=v$ outside a set of zero Sobolev
\p-capacity \cite{Sha}.

A function $u$ belongs to the \emph{local
  Newtonian space} $N^{1,p}_{\rm{loc}}(\Omega)$ if $u\in N^{1,p}(V)$
for all bounded open sets $V$ with $\overline{V}\subset\Omega$, the
latter space being defined by considering $V$ as a metric space with
the metric $d$ and the measure $\mu$ restricted to it.

Newtonian spaces share many properties of the classical Sobolev
spaces.  For example, if $u,v \in N^{1,p}_{\rm{loc}}(\Omega)$, then
$g_u=g_v$ a.e.\ in $\{x \in \Omega :\; u(x)=v(x)\}$, in particular
$g_{\min\{u,c\}}=g_u \chi_{\{u \neq c\}}$ for $c \in \R$.

\begin{remark}\label{absolute continuity}
  Note that as a consequence of the definition, the functions in
  $N^{1,p}(\Omega)$ are absolutely continuous on $p$-almost every
  path. This means that $u\circ \gamma$ is absolutely continuous on
  $[0,\textrm{length}(\gamma)]$ for $p$-almost every rectifiable
  arc-length parametrized path $\gamma$ in $\Omega$. This in turn
  implies that for each of these paths we have $|(u\circ
  \gamma)'(s)|\leq g(\gamma(s))$ for almost every $s\in
  [0,\textrm{length}(\gamma)]$. We refer to \cite[Theorem 1.56 and
  Lemma 2.14]{BjoBjo}.
\end{remark}

We shall also need a Newtonian space with zero boundary values. For a
measurable set $E\subset X$, let
\[
N_0^{1,p}(E) = \{f|_E:\; f\in N^{1,p}(X) \textrm{ and } f= 0 \textrm{ on
} X\setminus E\}.
\]
This space equipped with the norm inherited from $N^{1,p}(X)$ is a
Banach space.

We say that $X$ supports a \emph{weak $(1,p)$-Poincar\'e inequality}
if there exist constants $C_p>0$ and $\Lambda \geq 1$ such that for
all balls $B(x_0,r) \subset X$, all integrable functions $u$ on $X$
and all upper gradients $g$ of $u$,
\begin{equation} \label{PI} \vint_{B(x_0,r)} |u-u_B|\, d\mu \leq C_p
r\left(\vint_{B(x_0,\Lambda r)}g^p \, d\mu \right)^{1/p},
\end{equation}
where
\[
u_B:= \vint_{B(x_0,r)} u\, d\mu := \frac1{\mu({B(x_0,r)})}\int_{B(x_0,r)} u\, d\mu.
\]
If the metric measure space $X$ has not ``enough'' rectifiable paths,
it may happen that the continuous embedding $N^{1,p} \to L^p$, given
by the identity map, is onto. If $X$ has no nonconstant rectifiable
paths, then $g_u \equiv 0$ is the minimal $p$-weak upper gradient of
every function, and $N^{1,p}(X) = L^p(X)$ isometrically. The fact that
the Newtonian space is not simply $L^p(X)$ is implied, for instance,
by assuming that $X$ supports a weak $(1,p)$-Poincar\'e inequality.

\subsection{Parabolic setting}

Our set-up is the following. Let $\Om\subset X$ be an open set, and
$0<T<\infty$. We write $\Om_T := \Om_{(0,T)} := \Om\times (0,T)$ for a
space-time cylinder, and $z=(x,t)$ is a point in $\Om_T$. We denote by
$L^p(0,T; N^{1,p}(\Om))$ the parabolic space of functions
$u:\Om_T\to\R$ such that, for a.e. $t\in (0,T)$, $x\mapsto u(x,t)$
belongs to $N^{1,p}(\Om)$ and
\[
\int_{0}^{T}\|u\|_{N^{1,p}(\Omega)}^p\,dt< \infty,
\]
and similarly for $L\loc^p(0,T;N\loc^{1,p}(\Om))$. Here we have defined 
\[
g_u(x,t):=g_{u(\cdot,t)}(x)
\]
at $\nu$-almost every $(x,t)\in \Omega\times (0,T)$. 
 
The following calculus rules will be used throughout the text. Assume
$u,v \in L_{\textrm{loc}}^p(0,T;N_{\textrm{loc}}^{1,p}(\Omega))$. Then
for almost every $t$ and $\mu$-almost every $x$
\begin{align*}
&g_{u+v}\leq g_u+g_v,\\
&g_{uv}\leq |u|g_v+|v|g_u.
\end{align*}
In particular if $c$ is a constant, then $g_{c u}=|c| g_u$. For the
proof at each time level, see \cite{BjoBjo}. This proof guarantees
that $g_{u+v}$ and $g_{uv}$ are defined at almost every $t$ and
$\mu$-almost every $x$. The definition of the parabolic minimal
$p$-weak upper gradient then implies the result. Note that the above
does not claim that $uv$ is in the parabolic Newtonian space, even if
$u$ and $v$ are.

In the Euclidean case it can be shown that stating that a function
$u:\Omega \times (0,T) \to \R$, $u \in L^2_{\rm loc}(0,T; W^{1,2}_{\rm
  loc}(\Omega))$ is a weak solution to the doubly nonlinear parabolic
equation \eqref{equation}, is equivalent to stating that $u$ is fulfills
the variational problem
\begin{equation*}
\begin{split}
  p\int_{\textrm{supp}(\phi)} |u|^{p-2}u\frac{\partial \phi}{\partial t}
  \,dx \,dt+\int_{\textrm{supp}(\phi)} &|\grad u|^p\,dx\, dt \\ \leq
  &\int_{\textrm{supp}(\phi)} |\grad u+\grad \phi|^p\, dx\, dt
\end{split}
\end{equation*}
for every $\phi \in C_0^\infty (\Omega \times (0,T))$. Since partial
derivatives cannot be defined in a general metric space, there is
little sense in trying to define the weak formulation of the
equation~\eqref{equation} in the metric setting. The variational
approach on the other hand only considers integrals with absolute
values of partial derivatives and an inequality -- as opposed to
demanding a strict equation with gradients. This opens up the
possibility to extend the definition of a parabolic minimizer related
to the doubly nonlinear equation to metric measure spaces in the
following way:

\begin{deff}
  We say that a function $u\in L\loc^{p}(0,T; N\loc^{1,p}(\Om))$
  is a \emph{parabolic minimizer} if the inequality
\begin{equation} \label{minimizer}
 p \int_{\supp(\phi)}|u|^{p-2}u\frac{\partial \phi}{\partial t}\, d\nu  + \int_{\supp(\phi)}g_u^p\, d\nu  \leq
  \int_{\supp(\phi)}g_{u+\phi}^p\, d\nu
\end{equation}
holds for all $\phi \in \textrm{Lip}_0(\Om_T) = \{f\in\Lip(\Om_T):\;
\supp(f)\subset\Om_T\}$. If \eqref{minimizer} holds for all
nonnegative $\phi \in \textrm{Lip}_0(\Om_T)$ a function $u\in
L\loc^{p}(0,T;N\loc^{1,p}(\Om))$ is a \emph{parabolic superminimizer};
and a \emph{parabolic subminimizer} if \eqref{minimizer} holds for all
nonpositive $\phi \in \textrm{Lip}_0(\Om_T)$.
\end{deff}
Observe that here parabolic minimizers are scale invariant but not
translation invariant. 

Let $0<\alpha \leq 1$, let parameters $r$ and $T$ be positive, and
$t_0\in\R$. A space-time cylinder in $X\times\R$ is denoted by

\begin{align*}
Q_{\alpha r}(x,t)&=B(x,\alpha r)\times (t-T(\alpha r)^p, t+T(\alpha r)^p).
\end{align*}
It will also be of use to define positive and negative space-time
cylinders as
\begin{align*}
  \alpha Q^+(x,t)&=B(x,\alpha r)\times \left(t+T\left(\frac{1-\alpha}{2}\right)^pr^p, t+T\left( \frac{1+\alpha}{2}\right)^pr^p\right),\\
  \alpha Q^-(x,t)&=B(x,\alpha r)\times
  \left(t-T\left(\frac{1+\alpha}{2}\right)^pr^p, t-T\left(
      \frac{1-\alpha}{2}\right)^pr^p\right).
\end{align*}
Using these, we write
\begin{align*}
  &Q_{\alpha r}=Q_{\alpha r}(x_0, t_0),
\quad\alpha Q^+= \alpha Q^+(x_0, t_0), \quad\alpha Q^-= \alpha Q^-(x_0, t_0).
\end{align*}
Above $r$ is chosen according to $(x_0,t_0)$ and $T$ in such a way
that $Q_{r}\subset \Omega_T$. Our goal in this note is to prove the
following Harnack inequality using Moser's argument and energy
methods:

Suppose $1<p<\infty$ and assume that the measure $\mu$ in a geodesic
metric space $X$ is doubling with doubling constant $C_\mu$, and
supports a weak $(1,p)$-Poincar\'e inequality with constants $C_p$ and
$\Lambda$. Then a parabolic Harnack inequality is valid as follows.
Let $u>0$ be a parabolic minimizer in $Q_r\subset \Om_T$, locally
bounded and locally bounded away from zero. Let $0<\delta<1$. We have
\begin{equation}\label{weakHarnack} 
\esssup_{\delta Q^-}u \leq C\essinf_{\delta Q^+}u,
\end{equation}
where $C=C(C_\mu,C_p,\Lambda,\delta,p,T)$.

Note that the constant in the Harnack estimate does not depend on $r$ and so is scale invariant, as long as $r$ is such that $Q_r\subset \Omega_T$. The parameter $T$ controls only the  relative proportions of the spatial and time faces of $Q_r$.

\subsection{Sobolev-Poincar\'e inequalities}

We shall need the Sobolev inequality for functions with zero boundary
values; if $f\in N_0^{1,p}(B(x_0,R))$, then there exists a constant
$C>0$ only depending on $p, C_\mu$, and the constants $C_p$ and
$\Lambda$ in the Poincar\'e inequality, such that

\begin{equation} \label{Sobo} \left(\vint_{B(x_0,R)}|f|^{\kappa}\,
  d\mu \right)^{1/\kappa} \leq C R\left(\vint_{B(x_0,R)}g_f^p\, d\mu
  \right)^{1/p},
\end{equation}
where 
\begin{equation*} \label{eq:Sobexp} \kappa = \left\{ \begin{array}{ll}
      pq_\mu/(q_\mu-p), &  \textrm{if } 1<p<q_\mu \\ & \\
      \infty, & \textrm{if } p\geq q_\mu. \end{array} \right.
\end{equation*}
For this result we refer to \cite{KiSha}. The following weighted
version of the Poincar\'e inequality will also be needed.

\begin{lemma}\label{weighted Poincare} Let $f\in
N^{1,p}(B(x_0,R))$, and
\[
\phi(x) = \left(1-\frac{d(x,x_0)}{R}\right)^{\theta}_+,
\]
where $\theta>0$. Then there exists a positive constant $C=C(C_\mu,
C_p, p,\theta)$ such that
\begin{equation} \label{wPI} \vint_{B(x_0,r)}|f-f_\phi|^p\phi\, d\mu 
  \leq Cr^p \vint_{B(x_0, r)}g_f^p\phi\, d\mu 
\end{equation}
for all $0<r<R$, where
\[
f_\phi = \frac{\int_{B(x_0,r)}f\phi\, d\mu }{\int_{B(x_0,r)}\phi\, d\mu }.
\]
\begin{proof}[Sketch of proof]
  The main idea in the proof, for which we refer to
  Saloff-Coste~\cite[Theorem 5.3.4]{Saloff-Coste}, is to connect two
  points in the ball $B(x_0,r)$ with a certain finite chain of
  balls. For this chain we need to assume that our space $X$ is
  geodesic.
\end{proof}
\end{lemma}
\begin{remark}
  Lemma~\ref{weighted Poincare} will be used later in the proof of
  Lemma~\ref{measure estimate}. We stress, however, that apart from
  Lemma~\ref{weighted Poincare}, all other estimates prior to Lemma
  \ref{measure estimate} are valid without $X$ being geodesic.
\end{remark}

\subsection{Bombieri's and Giusti's abstract lemma}

A delicate step in the proof based on Moser's work is to use a
parabolic version of the John--Nirenberg inequality, i.e., exponential
integrability of BMO functions. To avoid the use of the parabolic BMO
class, the parabolic John--Nirenberg theorem is replaced with an
abstract lemma due to Bombieri and Giusti~\cite{BoGi}. Consult
\cite{Saloff-Coste} or \cite{KiKu} for the proof.

\begin{lemma} \label{abstractlemma} Let $\nu$ be a Borel measure and
  consider a collection of bounded measurable sets $U_\alpha$,
  $0<\alpha\leq 1$, with $U_{\alpha'}\subset U_\alpha$ if
  $\alpha'\leq\alpha$.

  Fix $0<\delta<1$, let $\theta, \gamma$, and $A$ be positive
  constants, and $0<q\leq \infty$. Moreover, if $q<\infty$, we assume
  that
\[
\nu(U_1)\leq A\nu(U_\delta)
\]
holds. Let $f$ be a positive measurable function on $U_1$ such that
for every $0<s\leq\min(1,q/2)$ we have
\[
\left(\vint_{U_{\alpha'}}f^q\,
d\nu\right)^{1/q}\leq\left(\frac{A}{(\alpha-\alpha')^\theta}\vint_{U_\alpha}f^s\,
d\nu\right)^{1/s}
\]
for every $\alpha,\alpha'$ such that
$0<\delta\leq\alpha'<\alpha\leq 1$. Assume further that $f$ satisfies
\[
\nu\left(\{x\in U_1:\; \log f>\lambda\}\right) \leq
\frac{A\nu(U_\delta)}{\lambda^\gamma}
\]
for all $\lambda>0$. Then
\[
\left(\vint_{U_\delta}f^q\, d\nu\right)^{1/q} \leq C,
\]
where $C=C(q,\delta,\theta,\gamma,A)$.
\end{lemma}

\section{Reverse H\"older inequalities for parabolic superminimizers}
\label{sect:super}

In this section, we prove an energy estimate for parabolic
superminimizers. After this, using the energy estimate we prove a
reverse H\"older inequality for negative powers of parabolic
superminimizers.

Establishing energy estimates for parabolic superminimizers is based
on substituting a suitably chosen test function into the
inequality~\eqref{minimizer}, and then performing partial integration
to extract the desired inequality from it. While doing this, we take
the time derivative of $u^{p-1}$, even though $u$ is not assumed to
have sufficient time regularity for this. Therefore, the reader should
consider the time derivation of $u$ as being formal. Justifications
for the formal treatment will be given in
Remark~\ref{timeRegularisation}.

\begin{lemma} \label{superCacc} Let $u> 0$ be a parabolic
  superminimizer, locally bounded away from zero, and $0<\eps\neq p-1$. Then
\begin{align*}
  \esssup_{0<t<T}&\int_\Om u^{p-1-\eps}\phi^p\, d\mu+\int_{\textrm{supp}(\phi)}   g_u^pu^{-1-\eps}\phi^p\, d\nu \\
  & \leq C_1\int_{\textrm{supp}(\phi)} u^{p-1-\eps}g_\phi^p\, d\nu+
  C_2\int_{\textrm{supp}(\phi)} u^{p-1-\eps}|(\phi^p)_t|\, d\nu
\end{align*}
for every $\phi\in \Lip_0(\Omega_T)$, $0\leq \phi \leq 1$, where
\begin{align*}
C_1  = \left(\frac{p}{\eps}\right)^p\left(1+ \frac{\eps |p-1-\eps|}{p(p-1)}\right), \quad 
\quad C_2  =\left(1+ \frac{p(p-1)}{\eps |p-1-\eps|}\right).
\end{align*}

\end{lemma}

\begin{proof} Assume $\varepsilon >0$, $\varepsilon \neq p-1$.  Let
  $\phi$ be a function $0\leq \phi \leq 1$, $\phi \in
  \textrm{Lip}_0(\Om_T)$. Since $\phi$ has compact support, we can choose $0<t_1<t_2<T$ such that $\phi(x,t)=0$ $\mu$-almost everywhere when $t\not \in(t_1,t_2)$.
 Since by assumption $u$ is locally bounded away from zero, we may assume a positive constant $\alpha>0$ such that after denoting $v=\alpha u$ we have $1-\varepsilon \phi^p
  v^{-\varepsilon -1} >0$ $\nu$-almost everywhere in the support of $\phi$. It then follows that $\nu$-almost everywhere in the support of $\phi$, we have
  \begin{align} \label{conservation of minus} g_{v+\phi^p v^{-\varepsilon}} &\leq
    (1-\eps\phi^p v^{-\eps-1})g_v+pv^{-\eps}\phi^{p-1}g_\phi.
\end{align}
That \eqref{conservation of minus} does indeed hold can be seen in the
following way: Let $t$ be such that $v(\cdot, t)\in
N^{1,p}(\Omega)$. Consider any arc-length parametrization $\gamma$ of
a rectifiable path on which $v(\cdot, t)$ is absolutely
continuous. Since $\phi(\cdot, t)$ is Lipschitz-continuous, it is
absolutely continuous on $\gamma$. Define
$h:[0,\textrm{length}(\gamma)]\rightarrow [0,\infty)$ by
\begin{align*}
 h(s)=v(\gamma(s),t)+\phi(\gamma (s),t)^ p v(\gamma(s),t)^ {-\varepsilon}.
\end{align*}
Then $h$ is absolutely continuous, and so  we have
\begin{align*}
  h'(s)=&(1-\varepsilon \phi(\gamma(s),t)^ p v(\gamma(s),t)^ {-\varepsilon -1}) \frac{\partial v(\gamma (s),t)}{\partial s}\\
  &+p \phi(\gamma(s),t)^{p-1}v(\gamma(s),t)^ {-\varepsilon}
  \frac{\partial \phi(\gamma(s),t)}{\partial s}
\end{align*}
for almost every $s\in [0,\textrm{length}(\gamma)]$ with respect to
the Lebesgue measure. We know that
\begin{align*}
  \left | \frac{\partial v(\gamma (s),t)}{\partial s}\right| \leq
  g_v(\gamma(s),t), \quad\left |\frac{\partial
      \phi(\gamma(s),t)}{\partial s}\right |\leq g_\phi (\gamma(s),t)
\end{align*}
for almost every $s\in [0, \textrm{length}(\gamma)]$. Hence 
\begin{align*}
  \left |\frac{\partial (v+\phi^p v^{-\varepsilon})(\gamma(s),t)}{\partial s}\right |&=|h'(s)|\\
  &\leq (1-\varepsilon \phi(\gamma(s),t)^p v(\gamma(s),t)^{-\varepsilon -1})g_v(\gamma(s),t)\\
  &+p\phi(\gamma(s),t)^{p-1}v(\gamma(s),t)^ {-\varepsilon}
  g_\phi(\gamma(s),t)
\end{align*}
for almost every $s\in [0, \textrm{length}(\gamma)]$. The fact that
this holds for $p$-almost every rectifiable path $\gamma$ now implies
\eqref{conservation of minus}. Using the convexity of the mapping $t\mapsto t^p$ we have
\begin{equation}\label{upper gradient on the right side}
\begin{split}
g_{v+\phi^p v^{-\varepsilon}}^p & \leq \left((1-\eps\phi^pv^{-\eps-1})g_v
+\eps\phi^pv^{-\eps-1}\frac{pv}{\eps\phi}g_\phi\right)^p \\
& \leq (1-\eps\phi^pv^{-\eps-1})g_v^p +p^p\eps^{1-p}v^{p-\eps-1}g_\phi^p. 
\end{split}
\end{equation}
Assume $0< \tau_1 < \tau_2 <T$. For a small enough $h>0$, denote
\begin{align*}
\chi_{[\tau_1,\tau_2]}^h(t)=
\begin{cases}
(t-\tau_1)/h,&\tau_1<t<t<\tau_1+h\\
1,& \tau_1+h\leq t \leq \tau_2-h \\
(\tau_2-t)/h,& \tau_2-h<t<\tau_2\\
0, &\textrm{otherwise}.
\end{cases}
\end{align*}
Integrating by parts, 
we find
\begin{align*}
  & \int_{\tau_1}^{\tau_2}\int_{\Omega}v^{p-1}(\phi^p v^{-\varepsilon} \chi_{[\tau_1, \tau_2]}^h)_t\, d\mu\,dt
  =\frac{(p-1)}{p-1-\varepsilon} \\&\,\cdot\left(\int_{\tau_1}^{\tau_2}\int_\Omega
    v^{p-1-\varepsilon}\phi^p (\chi^h_{[\tau_1,\tau_2]})_t\, d\mu\,dt   +\int_{\tau_1}^{\tau_2}\int_\Omega v^{p-1-\varepsilon} (\phi^p)_t \chi^h_{[\tau_1,\tau_2]}\,
    d\mu\,dt \right).
\end{align*}
After taking the limit $h\rightarrow 0$ in the expression above, and
using Lebesgue's differentiation theorem, we have
\begin{align*}
 & \lim_{h\rightarrow 0}\int_{\tau_1}^{\tau_2}\int_\Omega v^{p-1}(\phi^p v^{-\varepsilon} \chi_{[\tau_1,
  \tau_2]}^h)_t\, d\mu\,dt =\frac{(p-1)}{p-1-\varepsilon}\\
  &\quad  \cdot\left(-\left[\int_{\Omega\times \{t\}} v^{p-1-\varepsilon}\phi^p \, d\mu\right]_{t=\tau_1}^{\tau_2}  +\int_{\tau_1}^{\tau_2} \int_\Omega v^{p-1-\varepsilon} (\phi^p)_t\, d\mu\,dt \right).
\end{align*}
As $u$ is a positive parabolic superminimizer related to the doubly
nonlinear equation, also $v$ is a parabolic superminimizer.  Moreover,
by Remark~\ref{timeRegularisation} below, $\phi^p v^{-\varepsilon}
\chi_{[\tau_1, \tau_2]}^h$ is a nonnegative admissible test
function. Hence by the definition of a parabolic superminimizer and
\eqref{upper gradient on the right side} we have
\begin{align*}
  \begin{split}
    &\frac{p(p-1)}{p-1-\varepsilon} \left(-\left[\int_{\Omega\times
          \{t\}} v^{p-1-\varepsilon}\phi^p \,
        d\mu\right]_{t=\tau_1}^{\tau_2}
      +\int_{\tau_1}^{\tau_2}\int_\Omega v^{p-1-\varepsilon} (\phi^p)_t \, d\mu\,dt\right) \\
    &\leq \lim_{h\rightarrow 0} \left( -\int_{\textrm{supp}(\phi^p \chi_{[\tau_1,
  \tau_2]}^h)} g_v^p \, d\nu + \int_{\textrm{supp}(\phi^p \chi_{[\tau_1,
  \tau_2]}^h)} g_{v+\phi^p v^{-\varepsilon} \chi_{[\tau_1,
  \tau_2]}^h}^p \, d\nu \right) \\
    &\leq-\eps\int_{\textrm{supp}(\phi^p \chi_{[\tau_1,
  \tau_2]})} \phi^pv^{-\eps-1}g_v^p\, d\nu +
    p^p\eps^{1-p}\int_{\textrm{supp}(\phi^p \chi_{[\tau_1,
  \tau_2]})}v^{p-1-\eps}g_\phi^p\, d\nu.
  \end{split}
\end{align*}
 On one hand, setting $\tau_1=t_1$, and $\tau_2=t_2$, we obtain
\begin{equation}\label{estimate 1 for energy lemma}
\begin{split}
\int_{t_1}^{t_2}\int_{\Omega} \phi^pv^{-\eps-1}g_v^p\, d\mu \,dt\leq & \frac{p(p-1)}{\eps |p-1-\varepsilon |}\int_{t_1}^{t_2} \int_\Omega v^{p-1-\varepsilon} |(\phi^{p})_t| \, d\mu\, dt\\
&+\left(\frac{p}{\eps}\right)^p\int_{t_1}^{t_2} \int_\Omega v^{p-1-\eps}g_\phi^p\, d\mu\, dt.
\end{split}
\end{equation}
On the other hand, if $\varepsilon < p-1$, set $\tau_1=t$ and
$\tau_2=t_2$. If $\varepsilon > p-1$, set $\tau_1= t_1$ and
$\tau_2=t$.  We obtain
\begin{equation}\label{estimate 2 for energy lemma}
\begin{split}
 \qquad\frac{p(p-1)}{\varepsilon |p-1-\varepsilon |} &\int_{\Omega} v^{p-1-\varepsilon}(x,t) \phi^p(x,t) \, d\mu \\
  &\leq \frac{p(p-1)}{\eps |p-1-\varepsilon |}\int_{t_1}^{t_2}
  \int_\Omega v^{p-1-\varepsilon} |(\phi^{p})_t| \, d\mu\,
  dt \\
  &\qquad+\left(\frac{p}{\eps}\right)^p\int_{t_1}^{t_2} \int_\Omega
  v^{p-1-\eps}g_\phi^p\, d\mu\, dt.
  \end{split}
\end{equation}
This holds for almost every $t\in (t_1,t_2)$. Dividing \eqref{estimate
  2 for energy lemma} by $p(p-1)/\varepsilon|p-1-\varepsilon|$, and
adding the resulting expression to \eqref{estimate 1 for energy lemma}
yields the desired estimate for $v$, since the constants in the
inequality do not depend on $t\in(t_1,t_2)$ and $\phi$, $g_\phi$
vanish outside the support of $\phi$. The proof is completed by
dividing the resulting expression sidewise with the constant
$\alpha^{p-1-\varepsilon}$.
\end{proof}

\begin{remark}\label{timeRegularisation}
  We now give justifications for the formal treatment above. By a
  change of variable, it is straightforward to see that for a nonnegative
  parabolic super- or subminimizer $v$ and for an admissible test function $\psi$,
  for any small enough $s$, we have
\begin{align*} 
  p& \int_{\textrm{supp}(\psi)} v^{p-1}(x,t-s)\psi_t\, d\nu  + \int_{\textrm{supp}(\psi)} g_{v(x,t-s)}^p\, d\nu \\
  & \qquad \quad \leq
  \int_{\textrm{supp}(\psi)} g_{v(x,t-s)+\psi(x,t)}^p\, d\nu. 
\end{align*}
We multiply this inequality sidewise with a standard mollifier with
respect to the time variable $s$, and then integrate both sides with
respect to $s$. After using Fubini's theorem on the left side, this
yields
\begin{equation}\label{mollifiedMinimizingIneq}
\begin{split}
  p& \int_{\textrm{supp}(\psi)}(v^{p-1})_\sigma\psi_t\, d\nu +
  \left(\int_{\textrm{supp}(\psi)} g_{v}^p\,
    d\nu\right)_\sigma \\
  &\qquad \quad \leq \left(\int_{\textrm{supp}(\psi)}
    g_{v(x,t-s)+\psi(x,t)}^p\, d\nu \right)_\sigma.
    \end{split}
\end{equation} 
Here we have used the notation
\begin{align*}
  (v^{p-1})_\sigma(x,t)= \int_\R \theta_\sigma(s) v^{p-1}(x,t-s)\,ds,
\end{align*}
where $\theta$ is the standard mollifier and $\sigma>0$ is assumed to be
small enough so that everything stays in the time cylinder.  To be
precise, in the proof of Lemma \ref{superCacc} we then choose the test
function
\begin{align}\label{testfunctionSupCacc}
\psi=\phi^p ((v^{p-1})_\sigma) ^{-\varepsilon/(p-1)}\chi^h_{[\tau_1,\tau_2]}
\end{align}
with $\phi \in$Lip$_0(\Omega_{(t_1,t_2)})$.  The test function $\psi$
now has compact support and belongs to the space
$L^p(0,T;N^{1,p}(\Omega))$. By Lemma~2.7 in \cite{MMPP}, easily
adaptable for minimizers related to the doubly nonlinear equation,
$\psi$ can be plugged into the inequality
\eqref{mollifiedMinimizingIneq}. Similarly to the formal proof above,
partial integration is then performed to write the expression in a
form where $(v^{p-1})_\sigma$ is not differentiated with respect to
time. Once this is done we can take the limits $\sigma\rightarrow 0$
and $h \rightarrow 0$, which leads us back to the inequality above
\eqref{estimate 1 for energy lemma}. For details on justifying the
convergence of the upper gradient terms in
\eqref{mollifiedMinimizingIneq} as $\sigma\rightarrow 0$, we refer the
reader to \cite{MassSilj}.
\end{remark}

We prove next a reverse Hölder type inequality for negative powers of
parabolic superminimizers. The first step of the proof consists of
combining Sobolev's inequality with the energy esimate of
Lemma~\ref{superCacc}. Then, because the energy estimate is
homogeneous in powers, the obtained inequalities can be combined as in
Moser's iteration to complete the proof.

\begin{lemma} \label{weakHsuper1} Let $u>0$ be a parabolic
  superminimizer in $Q_r\subset \Om_T$, locally bounded away from zero
  and let $0<\delta<1$. Then there exist constants
  $C=C(C_\mu,C_p,\Lambda,p,\delta,T)$ and $\theta=\theta(C_\mu,p)$
  such that
\[
\essinf_{Q_{\alpha' r}}u \geq
\left(\frac{C}{(\alpha-\alpha')^\theta}\right)^{-1/q}
\left(\vint_{Q_{\alpha r}}u^{-q}\, d\nu \right)^{-1/q}
\]  
for every $0<\delta\leq\alpha'<\alpha\leq 1$ and for all $0<q\leq p$.
\end{lemma}

\begin{proof}
  Let us fix $\alpha',\alpha$ such that
  $0<\delta\leq\alpha'<\alpha\leq 1$, and divide the interval
  $(\alpha',\alpha)$ as follows: $\alpha_0 = \alpha$, $\alpha_\infty =
  \alpha'$, and
\[
\alpha_j = \alpha-(\alpha-\alpha')(1-\gamma^{-j}),
\]
where $\gamma = 2- p/\kappa = 1+ (\kappa-p)/\kappa >1$. We set
\[
Q_j = Q_{\alpha_j r} = B_j\times T_j = B(x_0,\alpha_j r) \times
(t_0-T(\alpha_jr)^p,t_0+T(\alpha_j r)^p),
\]
and choose the sequence of test-functions $\{\phi_j\}_{j=0}^{\infty}$
so that $\supp(\phi_j)\subset Q_j$, $0\leq\phi_j\leq 1$ on $Q_j$, and
$\phi_j=1$ in $Q_{j+1}$. Moreover, let each $\phi_j$ be such that
\begin{align*}
g_{\phi_j} \leq \frac{4\gamma^j}{(\alpha-\alpha')r} \quad
\textrm{and} \quad |(\phi_j)_t| \leq
\frac{1}{T}\left(\frac{4\gamma^j}{(\alpha-\alpha')r}\right)^p.
\end{align*}
Assume $\varepsilon>0$, $\varepsilon \neq p-1$. We have
\[
g_{u^{(p-1-\varepsilon)/p}\phi_j}^p \leq 2^{p-1}u^{p-1-\varepsilon}g_{\phi_j}^p+2^{p-1}\left(\frac{|p-1-\varepsilon|}{p}\right)^pu^{-\varepsilon-1}g_u^p\phi_j^p.
\]
Using H\"older's inequality brings us to the estimate
\begin{align*}
  &\vint_{T_{j+1}} \vint_{B_{j+1}}u^{(p-1-\varepsilon)\gamma}\, d\mu dt \\
  & \leq \vint_{T_{j+1}}\left(\vint_{B_{j+1}}
    u^{(p-1-\varepsilon)}\phi_j^p\,
    d\mu \right)^{(\kappa-p)/\kappa}\left(\vint_{B_{j+1}}(u^{(p-1-\varepsilon)/p}\phi_j)^\kappa\,d\mu \right)^{p/\kappa}\, dt \\
  & \leq \frac{|T_j|\mu(B_j)}{|T_{j+1}|\mu(B_{j+1})}\left(\esssup_{T_j}\vint_{B_j} u^{(p-1-\varepsilon)}\phi_j^p\,d\mu \right)^{(\kappa-p)/\kappa} \\
  & \qquad \qquad
  \cdot\vint_{T_j}\left(\vint_{B_j}(u^{(p-1-\varepsilon)/p}\phi_j)^\kappa\,d\mu
  \right)^{p/\kappa}\, dt.
\end{align*}
Observe that $|T_j| = 2T(\alpha_j r)^p$ and $\alpha_{j+1}\geq
\min\{\delta,(1+\gamma)^{-1}\}\alpha_j$. Thus the multiplicative
constant on the right-hand side is bounded by a constant independent
of $j,r,T, \alpha'$, and $\alpha$. We estimate the last term in the
preceding inequality by Sobolev's inequality (see \eqref{Sobo}). We
find
\begin{align*}
  & \left(\vint_{B_j}(u^{(p-1-\varepsilon)/p}\phi_j)^\kappa\,d\mu \right)^{p/\kappa} \leq Cr^p\vint_{B_j}g_{u^{(p-1-\varepsilon)/p}\phi_j}^p\,d\mu  \\
  & \quad \leq Cr^p\vint_{B_j}\left(u^{p-1-\varepsilon}g_{\phi_j}^p +
  \left(\frac{|p-1-\varepsilon|}{p}\right)^p u^{-\varepsilon-1}g_u^p\phi_j^p\right)\,
  d\mu ,
\end{align*}
where $C=C(C_\mu, C_p, \Lambda,p)$.
Since $\varepsilon>0$, $\varepsilon\neq p-1$, we may use Lemma~\ref{superCacc}
to obtain 
\begin{align*}
  &\vint_{T_{j+1}} \vint_{B_{j+1}}u^{(p-1-\varepsilon)\gamma}\, d\mu \,dt \leq C\left(\esssup_{T_j}\vint_{B_j} u^{p-1-\varepsilon}\phi_j^p\,d\mu \right)^{(\kappa-p)/\kappa} \\
  & \qquad \qquad \cdot\frac{C}{T\delta^p
  }\int_{T_j}\vint_{B_j}\left(u^{p-1-\varepsilon} g_{\phi_j}^p+
  \left(\frac{|p-1-\varepsilon|}{p}\right)^p u^{-\varepsilon-1}g_u^p\phi_j^p\right)\, d\mu \, dt \\
  & \leq
  C\left(\int_{T_j}\vint_{B_j}u^{p-1-\varepsilon}\left(C_1g_{\phi_j}^p+C_2|(\phi_j^p)_t|\right)\,
    d\mu \,dt \right)^{(\kappa-p)/\kappa} \\
  & \qquad \cdot\frac{C}{T\delta^p
  }\int_{T_j}\vint_{B_j}\left(u^{p-1-\varepsilon} g_{\phi_j}^p+
  |p-1-\varepsilon|^p u^{p-1-\varepsilon}(C_1 g_{\phi_j}^p+C_2|(\phi_j^p)_t|)\right)\, d\mu \, dt \\
  & \leq C(1+|p-1-\varepsilon|^p)
  \left(\frac{\gamma^{jp}}{(\alpha-\alpha')^p}\vint_{T_j}\vint_{B_j}u^{p-1-\varepsilon}\,
  d\mu \,dt\right)^\gamma,
\end{align*}
where $C=C(\varepsilon, C_\mu,C_p,\Lambda,p, \delta,T)$ is uniformly
bounded for every $\varepsilon$, except in the neighborhood of
$\varepsilon=0$.  
For each $j=0,1,\ldots\,$ we can now use the above estimate with
$\varepsilon_j\geq 2p-1$ chosen in such a way that
$p-1-\varepsilon_j=-p\gamma^j$, to write
\begin{equation}\label{iterative estimate}
\begin{split}
  &\left(\vint_{Q_{j+1}} u^{-p\gamma^{j+1}}\, d\nu
  \right)^{-1/p\gamma^{j+1}}
  \\
  & \geq (C2p^p
  \gamma^{pj})^{-1/p\gamma^{j+1}}\left(\frac{\gamma^{jp}}{(\alpha-\alpha')^p}\right)^{-1/p\gamma^j}\left(\vint_{Q_j}u^{-p\gamma^j}\,
    d\nu \right)^{-1/p\gamma^j},
\end{split}
\end{equation}
where $C=C(C_\mu,C_p,\Lambda,p, \delta,T)$. By iterating this, since
$\gamma>1$, we find that
\begin{align*}
  \essinf_{Q_\infty} u & \geq (Cp)^{\sum_{j=1}^\infty -1/\gamma^{j}} \gamma
  ^{\sum_{j=0}^\infty-(1+\gamma)j/\gamma^{j}}\left(\frac1{(\alpha-\alpha')}\right)^{\sum_{j=0}^\infty
    -1/\gamma^{j}}\left(\vint_{Q_0}u^{-p}\,
  d\nu \right)^{-1/p} \\
  &  =
 \left(\frac{C}{(\alpha-\alpha')}\right)^{-\gamma/(\gamma-1)}\left(\vint_{Q_0}u^{-p}\,
  d\nu \right)^{-1/p},
\end{align*}
where the constant $C=C(C_\mu,C_p,\Lambda,p,\delta,T)$ is positive and
finite.  The proof is now completed for any $0<q\leq p$ by using a
result from real analysis (see \cite[Theorem 3.38]{HeiKiMar}).
\end{proof}

We also prove a reverse H\"older inequality for positive powers of
parabolic superminimizers.

\begin{lemma} \label{weakRHsuper} Let $u>0$ be a parabolic
  superminimizer in $Q_r\subset \Omega_T$ which is locally bounded
  away from zero, and $0<\delta<1$. Then there exist constants $0<C =
  C(C_\mu, C_p,\Lambda,p,q,\delta,T)$ and $\theta=\theta(C_\mu,p)$
  such that
\[
\left(\vint_{Q_{\alpha'r} }u^q\, d\nu \right)^{1/q} \leq
\left(\frac{C}{(\alpha-\alpha')^\theta}\right)^{1/s}\left(\vint_{Q_{\alpha
      r}}u^s\, d\nu \right)^{1/s}
\]
for all $0<\delta\leq\alpha'<\alpha\leq 1$ and for all $0<s<q<(p-1)(2-p/\kappa)$ and $\kappa$ is as in \eqref{eq:Sobexp}.
\end{lemma}

\begin{proof} Assume $0<s<q<(p-1)(2-p/\kappa)$, where $\kappa$ is as
  in the Sobolev--Poincaré inequality. Then there exists a $k$ such
  that $s\gamma^{k-1}\leq q\leq s \gamma^k$. Let $\rho_0$ be such that
  $0<\rho_0\leq s$ and $q=\gamma^k\rho_0$. Now for each $j=0,...,k-1$,
  there exists a $0<\varepsilon_j<p-1$ such that
  $p-1-\varepsilon_j=\rho_0\gamma^j$. By the first part of the proof
  of the previous lemma, we have
\begin{equation*}
\begin{split}
  &\left(\vint_{Q_{j+1}}  u^{\rho_0\gamma^{j+1}}\, d\nu
  \right)^{1/\rho_0\gamma^{j+1}}
  \\
  & \leq (C2p^p
  \gamma^{pj})^{1/\rho_0\gamma^{j+1}}\left(\frac{\gamma^{pj}}{(\alpha-\alpha')^p}\right)^{1/\rho_0\gamma^j}\left(\vint_{Q_j}u^{\rho_0\gamma^j}\,
  d\nu \right)^{1/\rho_0\gamma^j},
\end{split}
\end{equation*}
where $C=C(C_\mu, C_p,\Lambda,p,q,\delta,T)$. Iterating this estimate for $j=0,\dots,k-1$ yields
\begin{equation}\label{finite iteration}
\begin{split}
  &\left(\vint_{Q_{\alpha'Q}}  u^{q}\, d\nu
  \right)^{1/q} \leq \left(\frac{C}{(\alpha-\alpha')^{\gamma^*}}\right)^{1/\rho_0}\left(\vint_{\alpha Q}u^{\rho_0}\,
  d\nu \right)^{1/\rho_0},
\end{split}
\end{equation}
where $C=C(C_\mu, C_p,\Lambda,p,q,\delta,T)$ blows up as $q$ tends to $(p-1)(2-p/\kappa)$ and 
\begin{align*}
\gamma^*=\frac{p\gamma}{\gamma-1}(1-\gamma^{-k}) \leq \frac{p\gamma}{\gamma-1}.
\end{align*}
 Using Hölder's inequality on the right-hand side  of \eqref{finite iteration}, setting $\theta=p\gamma/(\gamma-1)$ and using the fact that $s/\gamma \leq \rho_0 \leq s$ completes the proof.
\end{proof}

\section{Reverse H\"older inequalities for parabolic subminimizers}

In this section we prove estimates analogous to those in
Section~\ref{sect:super}, but this time for parabolic
subminimizers. This is done essentially identically to what was done for
superminimizers, but with a slight change in the test function we use. Then we utilize the obtained energy estimate to prove a reverse H\"older
inequality for positive powers of parabolic subminimizers. 

\begin{lemma} \label{subCacc} Let $u> 0$ be a locally bounded
  parabolic subminimizer and let $\eps\geq 1$. Then
\begin{align*}
  \esssup_{0<t<T}&\int_\Om u^{p-1+\eps}\phi^p\, d\mu+ \int_{\textrm{supp}(\phi)} u^{\eps-1}g_u\phi^p\, d\nu \\
  & \leq C_1\int_{\textrm{supp}(\phi)} u^{p-1+\eps}g_\phi^p\, d\nu +
  C_2\int_{\textrm{supp}(\phi)} u^{p-1+\eps}|(\phi^{p})_t|\, d\nu
\end{align*}
for every $\phi\in \Lip_0(\Omega_T)$, $0\leq \phi\leq 1$, where 
\begin{align*}
C_1  = \left(\frac{p}{\eps}\right)^p\left(1+ \frac{\eps |p-1+\eps|}{p(p-1)}\right), \quad 
\quad C_2  =\left(1+ \frac{p(p-1)}{\eps |p-1+\eps|}\right).
\end{align*}
\end{lemma}

\begin{proof}
  Let $0\leq \phi \leq 1$, $\phi \in \textrm{Lip}_0(\Omega_T)$, for
  some $0<t_1< t_2<T$. Let $\varepsilon>0$. Since by assumption $u$ is
  locally bounded, we can take a constant $\alpha>0$ such that after
  denoting $v=\alpha u$, we have $1-\varepsilon \phi^p
  v^{\varepsilon-1} >0$ almost everywhere in the support of
  $\phi$. Since $u$ is a subminimizer, also $v$ is a subminimizer and
  we can plug $ -\phi(x,t)^p v(x,t)^{\eps}\chi^h_{[\tau_1, \tau_2]}$
  as a test function into the inequality \eqref{minimizer}. The rest
  of the proof is now completely analogous to the proof of
  Lemma~\ref{superCacc}.
\end{proof}

We prove a reverse Hölder type inequality for positive powers of
parabolic subminimizers. Again, the proof consists of combining the
energy estimate of Lemma~\ref{subCacc} with Moser's iteration to
obtain the inequality.

\begin{lemma} \label{weakRHsub} Let $u> 0$ be a parabolic subminimizer
  in $Q_r\subset \Omega_T$ which is locally bounded and let
  $0<\delta<1$. Then there exist constants $C=C(C_\mu,C_p,\Lambda,
  p,\delta,T)$ and $\theta=\theta(C_\mu,p)$ such that the inequality
\[
\esssup_{Q_{\alpha' r}}u \leq
\left(\frac{C}{(\alpha-\alpha')^\theta}\right)^{1/q}\left(\vint_{Q_{\alpha
    r}}u^q\, d\nu \right)^{1/q}
\]  
holds for every $0<\delta\leq\alpha'<\alpha\leq 1$ and for all $0<q\leq p$.
\end{lemma}

\begin{proof}
  The steps of the proof are analogous to the proof of Lemma
  \ref{weakHsuper1}. The difference is that here we use
  Lemma~\ref{subCacc} and the observation that for each $\gamma^j$,
  $j=0,1,\dots$ there exists a $\varepsilon_j\geq 1$ such that
  $p-1+\varepsilon_j= p\gamma^j$.
%
\end{proof}

\section{Measure estimates for parabolic superminimizers}

The following logarithmic energy estimate will also be important to
our argument. Regarding the time derivation of $u^{p-1}$, the proof
presented below is again formal. Justifications for this can be given
as in Remark~\ref{timeRegularisation}; we use the test function as in
\eqref{testfunctionSupCacc}, but with $\varepsilon=p-1$.

\begin{lemma} \label{prelogsuperCacc} Let $u> 0$ be a parabolic
  superminimizer, locally bounded away from zero. Then the inequality
\begin{align*}
  \int_{\tau_1}^{\tau_2} \int_\Om g_{\log u}^p\phi^p\, d\mu dt -
  p\left[\int_{\Om \times \{t\}}\log u \phi^p\, d\mu
  \right]_{t=\tau_1}^{\tau_2}
  \\
  \leq\frac{p^p}{(p-1)^p}\int_{\tau_1}^{\tau_2} \int_\Om g_\phi^p\, d\mu dt +
  p\int_{\tau_1}^{\tau_2} \int_\Om|\log u||(\phi^p)_t|\, d\mu
  dt
\end{align*}
holds for every $\phi\in \textrm{Lip}_0(\Om_T)$, such that $0\leq
\phi\leq 1$ and almost every $0<\tau_1<\tau_2<T$.
\end{lemma}

\begin{proof}
  Let $\phi \in \Lip_0(\Om_{T})$ be such that $0\leq \phi \leq 1$. As
  in the preceding lemma, since both the definition of a parabolic
  superminimizer and the final weak Harnack inequality are scalable
  properties, we may assume that $u$ has been scaled in such a way
  that $1-(p-1) \phi^p u^{-p} >0$ almost everywhere in the support of
  $\phi$, and by using the convexity of the mapping $t\mapsto t^p$, we
  find
\begin{align*}
  g_{u+\phi^pu^{1-p}}^p &\leq \left( (1-(p-1)\phi^p u^{-p})g_u+p\phi^{p}(p-1) u^{-p}\frac{u}{\phi(p-1)}g_{\phi} \right)^p \\
  & \leq (1-(p-1)\phi^pu^{-p})g_u^p +p^p(p-1)^{1-p}g_\phi^p.
\end{align*}
Let $\chi_{[\tau_1, \tau_2]}^h$ be defined as in Lemma
\ref{superCacc}. Integrating by parts, we obtain
\begin{align*}
  &\int_{\tau_1}^{\tau_2}\int_\Omega u^{p-1}(\phi^p u^{1-p}\chi_{[\tau_1, \tau_2]}^h)_t\, d\mu\,dt  \\
  &=(p-1)\int_{\tau_1}^{\tau_2}\int_\Omega [(\log u) \phi^p(\chi^h_{[\tau_1,\tau_2]})_t +(\log u )(\phi^p)_t
  \chi^h_{[\tau_1,\tau_2]}]\,d\mu\,dt.
\end{align*}
Taking the limit $h\rightarrow 0$, we obtain by Lebesgue's theorem of
differentiation
\begin{align*}
&\lim_{h\rightarrow 0} \int_{\tau_1}^{\tau_2}\int_\Omega u^{p-1}(\phi^p u^{1-p}\chi_{[\tau_1, \tau_2]}^h)_t\, d\mu\,dt \\
= &(p-1)\left(-\left[\int_{\Omega \times \{t\}} (\log u )\phi^p \, d\mu\right]_{t=\tau_1}^{\tau_2} + \int_{\tau_1}^{\tau_2}\int_\Omega(\log u) (\phi^p)_t \,d\mu\,dt\right).
\end{align*}
As $u$ is a parabolic superminimizer and
$\phi^pu^{1-p}\chi^h_{[\tau_1, \tau_2]}$ is a nonnegative admissible
test-function, we obtain
\begin{align*}
  &p(p-1)\left(-\left[\int_{\Omega \times \{t\}} \log u \phi^p \, d\mu\right]_{t=\tau_1}^{\tau_2} +\int_{\tau_1}^{\tau_2}\int_\Omega(\log u) (\phi^p)_t \,d\mu\,dt\right) \\
  &\leq \lim_{h\rightarrow 0} \left( -\int_{\textrm{supp}(\phi^p\chi^h_{[\tau_1,
  \tau_2]})} g_u^p \, d\nu + \int_{\textrm{supp}(\phi^p\chi^h_{[\tau_1,
  \tau_2]})} g_{u+\phi^pu^{1-p}\chi^h_{[\tau_1,
  \tau_2]}}^p \, d\nu\right)\\
  &\leq -(p-1)\int_{\tau_1}^{\tau_2}\int_\Omega \phi^pg_{\log u}^p \,d\mu\,dt +p^p(p-1)^{1-p}\int_{\tau_1}^{\tau_2}\int_\Omega
  g_\phi^p\, d\mu\,dt.
\end{align*}
Rearranging terms completes the proof.
\end{proof}

Next, using the logarithmic energy estimate, we establish monotonicity
in time of the weighted integral of $\log u$. This in turn enables us
to estimate the measure of the level sets of $\log u$ around a time
level $t_0$.

\begin{lemma} \label{measure estimate} Let $u>0$ be a parabolic
  superminimizer in $Q_r\subset \Omega_T$ and assume $u$ is locally
  bounded away from zero. Let $0<\alpha<1$. Define
\[
\phi(x) = \left(1-2\frac{d(x,x_0)}{(1+\alpha) r}\right)_+,
\]
where $0<\alpha<1$ and $(x,t)\in Q_{ r}$. Let
\[
\beta = \frac1{N}\int_{B(x_0,r)}\log u(x,t_0)\phi^p(x)\, d\mu ,
\]
where 
\[
N=\int_{B(x_0,r)}\phi^p(x)\, d\mu .
\]
Then there exist positive constants $C=C(C_\mu, C_p, p,\alpha)$ and
$C'=C'( C_\mu, p,\alpha)$ such that
\[
\nu\left(\{(x,t)\in Q_{\alpha r}:\; t\leq t_0,\,\log u(x,t) > \lambda
+\beta+C'\}\right) \leq C\frac{\nu(Q_{\alpha r})}{\lambda^{p-1}}
\]
and
\[
\nu\left(\{(x,t)\in Q_{\alpha r}:\; t\geq t_0,\, \log u(x,t) < -\lambda
+\beta-C'\}\right) \leq C\frac{\nu(Q_{\alpha r})}{\lambda^{p-1}}
\]
for every $\lambda>0$.
\end{lemma}

\begin{proof}
  From the definition of $\phi$, it readily follows that $0\leq
  \phi\leq 1$, $g_\phi \leq (\alpha r)^{-1}$, and for every $t\in
  [t_0-T(\alpha r)^p, t_0+T(\alpha r)^p]$,
\begin{equation}\label{estimateforN}
\begin{split}
  &\left(\frac{1-\alpha}{2}\right)^p\mu(B(x_0,\alpha r)) \leq N \leq
  \mu(B(x_0,r)).
\end{split}
\end{equation}
We write
\[
v(x,t) = \log u(x,t)-\beta \qquad \textrm{and} \qquad V(t) =
\frac1{N}\int_{B(x_0,r)}v(x,t)\phi^p(x)\, d\mu ,
\]
and find that $V(t_0)=0$.  Let $0\leq \xi(t)\leq 1$ be a smooth
function such that supp$(\xi) \subset (t_0-Tr^p,t_0+Tr^p)$, and
$\xi(t)=1$ for all $t\in [t_0-T(\alpha r)^p, t_0+T(\alpha
r)^p]$. 

Write $\psi(x,t)=\phi(x)\xi(t)$.  Since $u$ is a positive
superminimizer bounded away from zero, we can use
Lemma~\ref{prelogsuperCacc} with $\psi$ as a test function. We obtain
for $t_0-T(\alpha r)^p< t_1<t_2 < t_0+T(\alpha r)^p$, since on this
interval $\xi(t)=1$,
\begin{align*}
\int_{t_1}^{t_2}\int_{B(x_0,r)}g_v^p\phi^p \, d\mu dt -p\left[NV(t)\right]_{t=t_1}^{t_2} &\leq  \frac{p^p}{(p-1)^p}\int_{t_1}^{t_2}\int_{B(x_0,r)}g_\phi^p\, d\mu dt \\
& \leq \frac{Cp^p}{(p-1)^p}(t_2-t_1)\frac{\mu(B(x_0,r))}{(\alpha r)^p},
\end{align*}
where $C=C(p)$. On the other hand, from the weighted Poincar\'e
inequality \eqref{wPI}, we have
\begin{align*}
 & \left(\frac{1-\alpha}{2}\right)^p\int_{t_1}^{t_2}\int_{B(x_0,\alpha r)}|v-V(t)|^p\, d\mu dt\\
  &\leq \int_{t_1}^{t_2}\int_{B(x_0,\alpha r)}|v-V(t)|^p\phi^p\, d\mu dt\leq Cr^p\int_{t_1}^{t_2}\int_{B(x_0,r)}g_v^p\phi^p\, d\mu dt,
\end{align*}
where $C=C(C_\mu,C_p,p,\alpha)$. By combining these we find
\begin{align*}
 & \frac{(1-\alpha)^p}{CNr^p}\int_{t_1}^{t_2}\int_{B(x_0,\alpha r)}  |v-V(t)|^p\, d\mu dt +V(t_1)-V(t_2) \\
  & \leq \frac{C(t_2-t_1)\mu(B(x_0,r))}{N(\alpha r)^p}
   \leq \left( \frac{2}{1-\alpha} \right)^p\frac{C(t_2-t_1)\mu(B(x_0,r))}{\mu(B(x_0,\alpha r))(\alpha r)^p}\\
  & \leq C' \frac{(t_2-t_1)}{(\alpha r)^p},
\end{align*}
where $C'=C'(C_\mu,C_p, p,\alpha)$. We denote
\[
w(x,t)=v(x,t)+\frac{C'(t-t_0)}{(\alpha r)^p} 
\]
and
\[
W(t)=V(t)+\frac{C'(t-t_0)}{(\alpha r)^p},
\]
and restate the preceding inequality as
\begin{equation*}
  \frac{(1-\alpha)^p}{CNr^p}\int_{t_1}^{t_2}\int_{B(x_0,\alpha r)}  |w-W(t)|^p\, d\mu dt +W(t_1)-W(t_2) \leq 0.
\end{equation*}
This implies that $W(t_1)\leq W(t_2)$ whenever $t_0-T(\alpha r)^p\leq
t_1<t_2 \leq t_0+T(\alpha r)^p$, i.e., the function $W$ is increasing,
thus differentiable for almost every $t \in (t_0-T(\alpha r)^p,
t_0+T(\alpha r)^p)$. As a consequence, we obtain
\begin{align}\label{derivestimate}
\frac{(1-\alpha)^p}{CNr^p}\int_{B(x_0,\alpha r)}  |w-W(t)|^p\, d\mu  -W'(t) \leq 0
\end{align}
for almost every $t_0-T(\alpha r)^p<t<t_0+T(\alpha
r)^p$. Let us denote
\begin{align*}
E_\lambda(t)&=\{x\in B(x_0,\alpha r):\; w(x,t)>\lambda\},\\
E_\lambda^-&= \{(x,t)\in Q_{\alpha r}:\; t< t_0,\, w(x,t)>\lambda\}.
\end{align*}
For every $t_0-T(\alpha r)^p< t< t_0$ and $\lambda>0$, since $W(t)\leq W(t_0)=0$, we have
\begin{align*}
(\lambda - W(t))^p \mu(E_\lambda^-(t)) \leq  \int_{B(x_0,\alpha r)} |w-W(t)|^p\, d\mu.
\end{align*} 
Hence we have
\begin{align*}
\frac{(1-\alpha)^p}{CNr^p} \mu(E_\lambda^-(t))-\frac{W'(t)}{(\lambda-W(t))^p}\leq 0
\end{align*}
for almost every $t_0- T(\alpha r)^p< t < t_0$. This yields, after
integrating over the interval $(t_0-T(\alpha r)^p, t_0)$,
\begin{align*}
\frac{\nu(E_\lambda^-)}{Nr^p} \leq \frac{C}{(1-\alpha)^{p}}  [(\lambda-W(t))^{-(p-1)} ]_{t=t_0-T(\alpha r)^p}^{t_0} \leq \frac{C}{(1-\alpha)^p\lambda^{p-1}},
\end{align*}
where $C=C(C_\mu,C_p,p,\alpha)$. Together with \eqref{estimateforN}, this implies
\begin{align*}
\nu\left(\{(x,t)\in Q_{\alpha r}:\; t\leq t_0,\, \log u(x,t) > \lambda
+\beta+C'\}\right) \leq C \frac{\nu(Q_{\alpha r})}{\lambda^{p-1}},
\end{align*}
where $C=C(C_\mu,C_p,p,\alpha)$. Denote then
\begin{align*}
E_\lambda^+(t) &= \{x\in B(x_0,\alpha r):\; w(x,t)<-\lambda\},\\
E_\lambda^+&= \{(x,t)\in Q_{\alpha r}:\; t> t_0,\, w(x,t)<-\lambda\}.
\end{align*}
Similarly to the case of $E_\lambda^-$, using the monotonicity of $W(t)$, we obtain
\begin{align*}
  (\lambda + W(t))^p \mu(E_\lambda^+(t)) \leq \int_{B(x_0,\alpha r)}
  |w-W(t)|^p\, d\mu
\end{align*} 
for every $t_0<t<t_0+T(\alpha r)^p$. This together with
\eqref{derivestimate} leads to
\begin{align*}
\frac{(1-\alpha)^p\mu(E_\lambda^+(t))}{CNr^p}-\frac{W'(t)}{(\lambda+W(t))^p}\leq 0
\end{align*}
for almost every $t_0 < t < t_0+ T(\alpha r)^p$. Integration over the
interval $(t_0, t_0+T(\alpha r)^p)$ gives now
\begin{align*}
\frac{\nu( E_\lambda^+)}{ N r^p} \leq - \frac{C}{(1-\alpha)^{p}}[ (\lambda+ W(t))^{-(p-1)}]_{t=t_0}^{t_0+T(\alpha r)^p} \leq \frac{C}{(1-\alpha)^{p} \lambda^{p-1}},
\end{align*}
and thus after using \eqref{estimateforN} we may conclude
\begin{align*}
\nu(\{(x,t)\in Q_{\alpha r}:\; t\geq t_0,\, \log u<-\lambda+\beta-C'\}) \leq C \frac{\nu( Q_{\alpha r})}{ \lambda^{p-1}}.
\end{align*}
Again $C=C(C_\mu, C_p, p,\alpha)$.
\end{proof}

\section{Harnack's inequality for parabolic minimizers}

Having established a logarithmic measure estimate for superminimizers around a
time level $t_0$, we have the prerequisites to use
Lemma~\ref{abstractlemma}. This way for parabolic superminimizers we can glue the reverse Hölder inequality  for negative powers together with the reverse Hölder inequality for positive powers. We obtain a weak form of the Harnack inequality for parabolic superminimizers locally bounded away from zero. This result is in some sense finer than the final Harnack inequality since we only assume the superminimizing property, and hence it is of interest in itself.  Observe in the following how, from applying
Lemma~\ref{abstractlemma} separately on both sides of the time level
$t_0$, a waiting time inevitably appears between the negative and
positive time segments.

\begin{lemma} \label{weakHsuper2} Let $u>0$ be a parabolic
  superminimizer in $Q_r\subset\Om_T$ which is bounded away from
  zero. Then
\[
\left(\vint_{\delta Q^-} u^q\, d\nu \right)^{1/q} \leq C\essinf_{\delta Q^+}u,
\]
where $0<\delta<1$ and $0<q<(p-1)(2-p/\kappa)$. Here $C=C(C_\mu, C_p,
\Lambda, p,q,\delta,T)$.
\begin{proof}
  Assume $0<\delta< 1$. Let $\beta$ and $C'$ be as in
  Lemma~\ref{measure estimate}. By Lemma~\ref{weakHsuper1} there
  exists a positive constant $C=C(C_\mu,C_p,\Lambda,p,\delta,T)$, such
  that for every $0<s\leq p$ and $0< \delta \leq\alpha'< \alpha <1$,
  we have
  \begin{equation}\label{first condition}
  \begin{split}
    (\esssup_{\alpha' Q^+} u^{-1}e^{\beta-C'})^{-1}&=\essinf_{\alpha' Q^+} u
    e^{-\beta+C'}\\&\geq
    C\left(\frac1{(\alpha-\alpha')^\theta}\vint_{\alpha
      Q^+}(ue^{-\beta+C'})^{-s}\, d\nu \right)^{-1/s}.
  \end{split}
  \end{equation}
  By Lemma~\ref{measure estimate} applied to
  $\{Q_{(3+\delta)r/4}:\; t\geq t_0\}$, we have
\begin{equation}\label{second condition}
  \begin{split}
    & \nu\left(\{(x,t)\in \frac{1+\delta}{2} Q^+:\; \log (u^{-1}e^{\beta-C'}) > \lambda \}\right) \\
    &\leq \nu\left(\{(x,t)\in Q_{(3+\delta)r/4}:\; t\geq t_0,\, \log (u^{-1}e^{\beta-C'}) > \lambda \}\right) \\
    &\leq C\frac{\nu(Q_{(3+\delta)r/4})}{\lambda^{p-1}}\leq
    C\frac{\nu(\delta Q^+)}{\lambda^{p-1}}
\end{split}
\end{equation}
for every $\lambda>0$. In the last step of the above inequality, we
used the doubling property of $\mu$, and so $C=C(C_\mu,
C_p,p,\delta)$. From \eqref{first condition} and \eqref{second
  condition}, we now see that the conditions of Lemma
\ref{abstractlemma}, with $(1+\delta)/2Q^+$ in place of $U_1$, are
met. Hence
\begin{align}\label{supestimate}
 \esssup_{\delta Q^+} u^{-1}e^{\beta-C'} \leq C,
\end{align}
where $C=C(C_\mu, C_p,\Lambda,p,\delta,T)$. 
From Lemma~\ref{weakRHsuper} we know there exists a positive constant
$C = C(C_\mu,C_p, \Lambda,p,q,\delta,T)$ for which
 \[
 \left(\vint_{\alpha' Q^-} (u
  e^{-\beta-C'})^q \, d\nu \right)^{1/q} \leq
 \left(\frac{C}{(\alpha-\alpha')^\theta}\right)^{1/s}\left(\vint_{\alpha
     Q^-}(u
  e^{-\beta-C'})^s\, d\nu \right)^{1/s}
\]
for every $0\leq\delta <\alpha'<\alpha\leq 1$ and for all
$0<s<q<(p-1)(2-p/\kappa)$. 
Moreover for $\delta Q^-$, since $u$ is a positive superminimizer
bounded away from zero, we can use Lemma \ref{measure estimate} to get
\[
\nu\left(\{(x,t)\in \frac{1+\delta}{2} Q^-:\; \log (u
  e^{-\beta-C'}) > \lambda
\}\right) \leq C\frac{\nu(\delta Q^-)}{\lambda^{p-1}}.
\]
Therefore, by Lemma~\ref{abstractlemma} we have
\begin{equation} \label{eq:uniformubound}
\left(\vint_{\delta Q^-}(u
  e^{-\beta-C'})^q\, d\nu\right)^{1/q} \leq C,
\end{equation}
where $C=C(C_\mu, C_p,\Lambda,p,q,\delta,T)$. Multiplying
\eqref{eq:uniformubound} with \eqref{supestimate} gives the result
\begin{align*}
  \left( \vint_{\delta Q^-} u^q\, d\mu dt\right)^{1/q} \leq C
  \essinf_{\delta Q^+} u
\end{align*}
for every $0<q<(p-1)(2-p/\kappa)$, where $C=C(C_\mu, C_p,\Lambda,p,q,\delta,T)$.
\end{proof}
\end{lemma}

We end this paper by completing the proof of Harnack's inequality for
parabolic minimizers. This is the first point at which we make use of
the fact that a minimizer is both a sub- and superminimizer.

\begin{theorem} \label{thm:Harnack} Suppose $1<p<\infty$ and assume
  that the measure $\mu$ in a geodesic metric space $X$ is doubling
  with doubling constant $C_\mu$, and the space supports a weak
  $(1,p)$-Poincar\'e inequality with constants $C_p$ and
  $\Lambda$. Then a parabolic Harnack inequality is valid as follows:
  Let $u>0$ be a parabolic minimizer in $Q_r\subset \Om_T$ which is
  locally bounded away from zero, and locally bounded. Let
  $0<\delta<1$. Then
\begin{equation*} \label{Harnack}
\esssup_{\delta Q^-}u \leq C\essinf_{\delta Q^+}u,
\end{equation*}
where $0<C<\infty$ and $C=C(C_\mu,C_p, \Lambda,p,\delta, T)$.

\begin{proof}
  By assumption, $u$ is both a parabolic sub- and
  superminimizer. Hence we may combine Lemma~\ref{weakRHsub} with
  Lemma~\ref{weakHsuper2} to obtain
\begin{align*}
  \esssup_{\delta Q^-}u &\leq \left(\frac{C}{((1+\delta)/2-\delta)^\theta}\right)^{1/(p-1)}\left(\vint_{(1+\delta)/2Q^-}u^{p-1}\, d\nu \right)^{1/(p-1)}\\
  &\leq C \essinf_{(1+\delta)/2 Q^+} u \leq C \essinf_{\delta
    Q^+} u,
\end{align*}
Where $\theta=\theta(C_\mu,p)$ and so $C=C(C_\mu, C_p,\Lambda, p,\delta,
T)$. 
\end{proof}

\end{theorem}

\end{document}